\documentclass[journal]{IEEEtran}
\IEEEoverridecommandlockouts                              
\bstctlcite{IEEEexample:BSTcontrol}                                                          
\usepackage{flushend} 
\usepackage{romannum}
\usepackage[usenames,dvipsnames]{xcolor}
\usepackage{tkz-berge}
\usetikzlibrary{fit,shapes}                                     
\usepackage{calrsfs}
\DeclareMathAlphabet{\pazocal}{OMS}{zplm}{m}{n}
\usepackage{graphicx}
\usepackage{graphics} 
\usepackage{epsfig} 
\usepackage{mathptmx}
\usepackage{caption}
\usepackage{subcaption}
\usepackage{times} 
\usepackage{tikz}
\usetikzlibrary{positioning}
\usepackage{amsmath, amssymb}

\usepackage[perpage, symbol]{footmisc}
\usepackage{amsthm}
\usepackage{amsfonts} 
\usepackage{setspace}
\usepackage{multirow}
\usepackage{textcomp}
\usepackage[english]{babel}
\usepackage{float} 
\usepackage{algorithmicx}
\usepackage[section]{algorithm}
\usepackage{algpascal}
\usepackage{algc}
\usepackage{algcompatible}
\usepackage{algpseudocode}
\let\oldReturn\Return
\renewcommand{\Return}{\State\oldReturn}
\usepackage{mathtools}
\usepackage{bm}
\usepackage{mathptmx}
\usepackage{times} 
\usepackage{mathtools, cuted}
\usepackage{textcomp}
\usepackage[english]{babel}
\usepackage{rotating}
\usepackage{pgfplots}
\pgfplotsset{compat=1.5}
\usepackage{siunitx}
\usepackage{pgfplotstable}
\usepackage{filecontents}

\newcommand{\B}{\pazocal{B}}

\newcommand{\K}{\pazocal{K}}
\newcommand{\W}{\pazocal{W}}
\newcommand{\D}{\pazocal{D}}
\newcommand{\E}{\pazocal{E}}
\newcommand{\X}{\pazocal{X}}
\newcommand{\U}{\pazocal{U}}
\newcommand{\F}{\pazocal{F}}
\newcommand{\J}{\pazocal{J}}
\newcommand{\I}{\pazocal{I}}

\renewcommand{\S}{\pazocal{S}}
\newcommand{\x}{{X}}
\renewcommand{\u}{{U}}
\newcommand{\dx}{{X'}}
\newcommand{\remove}[1]{}

\def \cN{{\mathcal N}}

\def \*{\star}
\def \10n{\!\!\!\!\!\!\!\!\!\!}

\def \FAB {{\F(\bA,\bB)}}
\def \Fc {{\F(\bA,\bB,c)}}

\def \c {{C_{\rm OPT}}}
\def \flp {{f^\*_{\rm LP}}}

\newcommand{\pl}{{\varphi_{\min}}}

\newcommand{\R}{\mathbb{R}}
\newcommand{\bA}{\bar{A}}
\newcommand{\bB}{\bar{B}}
\newcommand{\bC}{\bar{C}}

\newtheorem{theorem}{Theorem}
\newtheorem{lem}[theorem]{Lemma}

\newtheorem{defn}[theorem]{Definition}
\newtheorem{cor}[theorem]{Corollary}
\newtheorem{rem}[theorem]{Remark}
\newtheorem{prob}[theorem]{Problem}

\newtheorem{prop}[theorem]{Proposition}
\numberwithin{theorem}{section}

\graphicspath{{Figures/}}

\title{\LARGE \bf Approximating Constrained Minimum Input Selection for State Space Structural Controllability}
\author{Shana~Moothedath,
        Prasanna~Chaporkar
        and~Madhu~N.~Belur
\thanks{The authors are in the Department of Electrical Engineering, Indian Institute of Technology Bombay, India. Email: $\lbrace$shana, chaporkar, belur$\rbrace$@ee.iitb.ac.in. This work was supported in part by SERB (DST) and BRNS, India.}}

\begin{document}
\maketitle
\thispagestyle{empty}
\pagestyle{empty}

\begin{abstract}
This paper looks at two problems, minimum constrained input selection and minimum cost constrained input selection for state space structured systems. The input matrix is constrained in the sense that the set of states that each input can influence is pre-specified and each input has a cost associated with it. Our goal is to optimally select an input set from the set of inputs given that the system is controllable. These problems are known to be NP-hard. Firstly, we give a new necessary and sufficient graph theoretic condition for checking structural controllability using flow networks. Using this condition we give a polynomial reduction of both these problems to a known NP-hard problem, the minimum cost fixed flow problem (MCFF). Subsequently, we prove that an optimal solution to the MCFF problem corresponds to an optimal solution to the original controllability problem. We also show that approximation schemes of MCFF directly applies to minimum cost constrained input selection problems. Using the special structure of the flow network constructed for the structured system, we give a polynomial approximation algorithm based on minimum weight bipartite matching and a greedy selection scheme for solving MCFF on system flow network. The proposed algorithm gives a $\Delta$-approximate solution to MCFF, where $\Delta$ denotes the maximum in-degree of input vertices in the flow network of the structured system.
\end{abstract}

\section{Introduction} \label{sec:intro}
We consider a control system $\dot{x} = Ax + Bu$, $y = Cx$, where $A$ is a state matrix, $B$ is an input matrix, $C$ is an output matrix, $x$ is a state vector, $u$ is a vector of inputs and $y$ is a vector of outputs. We assume that the exact entries of $A$, $B$ and $C$ are not known, rather only the location of the zero entries is known. Each input and output has a cost associated with it. Our aim in this paper is to choose a subset of inputs (outputs, resp.) that keeps the system controllable (observable, resp.) while minimizing the cost. We motivate the problem with the following examples.

\subsection{RLC Network Example}
Consider an electric network with cascaded series and parallel RLC circuits. Inputs to the network are voltage and current sources and the states are the capacitor voltages and inductor currents. Controllability of the network is about using the network inputs to drive the system from arbitrary initial state to arbitrary final state in a suitable finite time. A network is said to be controllable if it is possible to achieve arbitrary specified voltages across the capacitor terminals and currents through the inductors. In this context, given permissible locations for including voltage and current sources, it is a legitimate question to ask for the minimum number of sources to include in the network such that it is controllable. In addition, assume that there is a cost associated with each source based on installation, monitoring, reliability and so on. Then it is useful to find the set of sources needed to control the whole network such that the total cost incurred is minimum. The conventional method of checking controllability of the above described network is using Kalman's rank criterion and this requires the information about the state matrix $A$ and input matrix $B$. The entries of the state matrix depends not only on the connections of the network but also on the numerical values of its circuit components. Practically, the numerical values of the circuit components are not known accurately. Their actual values lie within some tolerance of the specified values. In addition, due to the environmental effects and ageing the values of the circuit components deviate from the specified values. Taking into consideration the above facts, practically the exact values of most of the entries of the state matrix are not known. However, certain entries are precisely known and this mostly happens for the zero entries. A zero entry becomes non-zero if a new connection is made in the network and this happens with the knowledge of the network designer. Hence, we assume that the numerical values of the circuit components are not known, but the presence of links or connections in the network are known. More precisely, the state matrix is not known, but the zero and non-zero pattern (sparsity pattern) of the state matrix is known. Our aim is to find the minimum set of sources and the minimum cost incurring set of sources to include in the network for `almost' all realizations of $A$ and $B$ matrices having the pre-specified sparsity pattern such that the network is controllable.

\subsection{Complex Network Example}
Study of complex networks has achieved broad research interest in the last decade because of its wide range of applications in various fields, including biology, chemistry, economics, technology and electrical. To motivate the problems considered in this paper, we explain one such application here. Consider a social network consisting of players where each player is linked to a set of players under relationships of various types including friendship, kinship, official and political. The nature and importance of relation of a player decides the link weights of his connections to other players in the graph. Our aim is to find the minimum set of inputs (where each input can influence a pre-specified set of players) from the given feasible input set such that the whole system is controllable. In this context, the system is said to be controllable if each player in the graph is controllable. In addition, suppose each input is associated with a cost depending on the nature of the information to be passed in the network. Then our aim is to find the set of inputs that control the network by incurring minimum cost. The costs associated with inputs play a vital role in selection preferences of various applications like (a) leader selection where certain agents are preferred over others for performing some specific task \cite{MesEge:10} or (b) optimal selection of key players in a social network for circulating information where the preference depends on the nature of the information \cite{Bor:06}. Considering enormous sizes of the complex networks often the link weights of the graph are not known, but the presence of links is known. Even if the link weights are known, finding an optimal set of inputs by brute-force requires checking controllability of $2^n -1$ distinct combinations, where $n$ is the system dimension. Complex networks are of large dimensions and hence using this method for finding optimal input set is not feasible. Summarizing, the massive size of these systems makes it a challenging task to find a minimum set of inputs to control the system. Additionally, if the possible set of inputs is also given, then the problem becomes constrained making it more challenging.

Lin proved that controllability and observability can be verified generically even if the numerical entries of matrices $A, B$ and $C$ are not known, but their sparsity patterns are known \cite{Lin:74}. The study of system properties like controllability and observability using the sparsity patterns of system matrices instead of the matrices themselves is called  as {\it structural controllability} and {\it structural observability} respectively. However, because of duality between controllability and observability in linear time invariant systems, method for solving one automatically leads to a method for solving the other \cite{Kai:80}. So, we will discuss only the controllability problem in this paper. The observability counter part follows by duality.

Structural controllability is a well studied problem and graph theoretic formulations of its variants are available. They mainly use concepts of bipartite matching and graph connectivity, see \cite{LiuBar:16} and references therein for more details. For instance, the problem of identifying the minimum number of inputs required to achieve structural controllability is considered and conditions using maximum matching are given in \cite{ConDioWou:02}, \cite{LiuSloJeaBar:11}. Here structure of input matrix is not specified which makes the problem polynomial complexity. Another variant, input addition for structural controllability is studied using graph theoretic tools like Dulmage-Mendelsohn (DM) decomposition and matching in \cite{ComDio:13}. Similarly, necessary and sufficient conditions for strong structural controllability is given using the concept of constrained matchings in \cite{ChaMes:13}. However, optimal selection of inputs for structural controllability, referred to as {\it minimum input design} problem, is the most addressed variant in literature.

Given the sparsity pattern of the state matrix, the minimum input design problem aims at finding the sparsest input matrix that can control the system. Polynomial algorithms of complexities $O({\ell} n^{1.5}), O(n^3)$ and $O(n + \ell\sqrt{n})$ are given in \cite{PeqKarAgu:13}, \cite{PeqKarAgu_2:16} and \cite{Ols:15} respectively for solving this problem, where $\ell$ denotes the number of non-zero entries in the state matrix and $n$ is the number of states. Some papers on minimum input design focus on dedicated input control, where every input can directly control a single state only \cite{Ols:15}, \cite{PeqKarAgu:13}. Note that here the input matrix is diagonal. Another problem studied in literature is finding a single input vector with minimum number of non-zero entries such that the given structured system is controllable. An explicit
characterization of the solution of the minimum input design problem when the input matrix is of dimension one is given in \cite{ComDio:15}. Minimum cost input design problem is studied in \cite{PeqKarAgu:16}. The objective here is to find an input matrix that can control the system as well as incur minimum cost when each state is associated with a cost. Thus given the structure of state matrix, finding an input matrix with minimum number of non-zero entries or that incur minimum cost is polynomially solvable and efficient algorithms are available. However, if the input matrix is specified and the aim is to select an optimal input set, then the minimum controllability problems are NP-hard \cite{PeqSouPed:15}. We refer to these problems as {\it minimum constrained input selection} problems.

 The NP-hardness result of the {\it minimum constrained input selection} problem is given in \cite{PeqSouPed:15}. Consequently, the {\it minimum cost constrained input selection} problem also turns out to be NP-hard.  Here the aim is to find an optimal set of input that control the system and incur in minimum cost when each input has a cost associated with it. However, if the state bipartite graph (see Section \ref{sec:graph} for more details) has a perfect matching and the inputs are dedicated i.e., diagonal input matrix, then minimum constrained input selection problem is not NP-hard and this case is considered in \cite{PeqSouPed:15}. Similarly, if the state digraph (see Section \ref{sec:graph} for more details) is irreducible (a digraph is said to be irreducible if there exists a directed path between every distinct pair of nodes), then the minimum cost constrained input selection problem is no longer NP-hard and this case is considered in \cite{PeqKarPap:15}. However, minimum constrained input selection problems are not addressed in their full generality. Reducing these problems to the standard NP-hard problems with good approximation schemes was posed as an open problem in \cite{PeqSouPed:15}. In this paper, {\it  we reduce these problems to the minimum cost fixed flow problem and present a polynomial approximation algorithm for approximating the minimum cost fixed flow problem on the system flow network. Note that in this work we did not impose any assumption on the structured system and both the problems are considered in their full generality.}

Summarizing, this paper develops an algorithm for finding a minimum (in the sense of cost) input set for structural controllability when the sparsity pattern of input matrix $B$ is specified. Henceforth, we will discuss structural controllability problems, namely minimum constrained input selection and minimum cost constrained input selection. Our key contributions are presented below.

\noindent
$\bullet$ We give a new graph theoretic necessary and sufficient condition for checking structural controllability using flow networks (see Theorem \ref{th:maxflow}).

\noindent
$\bullet$ We reduce the minimum {\it cost} constrained input selection problem to a minimum cost fixed flow problem in polynomial time.

\noindent
$\bullet$ We prove that an optimal solution to the minimum cost fixed flow problem corresponds to an optimal solution to the minimum cost constrained input selection problem (see Theorem \ref{th:mcffformulation1}).

\noindent
$\bullet$ We prove that approximation schemes for minimum cost fixed flow problem applies to minimum cost constrained input selection problem ()see Theorem~\ref{th:epsilon1}.

\noindent
$\bullet$ We give an $\Delta$-approximation to the minimum cost constrained input selection problem using minimum cost flow problem with polynomial complexity (see Theorem \ref{th:approximation}).

\noindent
$\bullet$ Using duality between controllability and observability in linear time invariant systems, we extend all results in this paper to {\it minimum cost constrained output selection} problem and {\it minimum constrained output selection} problem (see Remark~\ref{rem:obs1}).

\noindent 
$\bullet$ All the analysis and results given in this paper directly extends to discrete time systems because of the same controllability criterion for continuous and discrete systems (see Remark \ref{rem:discrete}).

The organization of this paper is as follows: in Section~\ref{sec:problem}, we detail the two structural controllability problems considered in this paper. In Section~\ref{sec:graph}, we explain structural controllability using concepts from graph theory. In Section~\ref{sec:equi}, we discuss a relation between structural controllability and maximum flow problem. Using this a new graph theoretic condition for checking structural controllability is also given. In Section~\ref{sec:mcff}, formulation of minimum cost constrained input selection  problem as a flow problem, called minimum cost flow problem, is given. In Section~\ref{sec:approx}, we formulate a linear programming problem for solving the structural controllability problems. Using this a $\Delta$-approximate solution is given. In Section~\ref{sec:cases}, few special classes of structured systems are considered and approximation results for theses classes is presented. Finally, Section~\ref{sec:conclu} gives the concluding remarks.
\section{Problem Formulation}\label{sec:problem}

Let $\bA$ ($\bB$, resp.) be $n \times n$ ($n \times m$, resp.) matrix whose entries are either $\*$ or $0$. We say that $\bA$ and $\bB$ {\it structurally represent} state and input matrices of any control system $\dot{x} = Ax + Bu$ where $A$ and $B$ satisfy:
\begin{eqnarray}\label{eq:struc}
A_{ij} &=& 0 \mbox{~whenever~} \bA_{ij} = 0,\mbox{~and} \nonumber \\
B_{ij} &=& 0 \mbox{~whenever~} \bB_{ij} = 0.
\end{eqnarray}
Note that non-zero entries of $A$ and $B$ can occur only at places where $\bA$ and $\bB$ respectively have $\*$. We refer to matrices $A$ and $B$ that satisfy \eqref{eq:struc} as a {\it numerical realization} of $\bA$ and $\bB$ respectively and $(\bA, \bB)$ as a {\it structured system}. Thus, $(\bA, \bB)$ structurally represents a class of control systems corresponding to all possible numerical realizations. Key idea in structural controllability is to determine controllability of the class of control systems represented by $(\bA, \bB)$. Specifically, we have the following definition.

\begin{defn}\label{def:struccont}
The structured system $(\bA, \bB)$ is said to be structurally controllable if there exists a numerical realization $(A, B)$ such that $(A, B)$ is controllable.
\end{defn}

Even though the definition of structural controllability requires only one controllable realization, it is shown that if a system is structurally controllable then `almost all' numerical realizations of the same structure is controllable  \cite{Rei:88}, \cite{Mur:87}. That is, given a structurally controllable $(\bA, \bB)$, the set of all uncontrollable realizations $(A, B)$ has Lebesgue measure zero. Thus structural controllability is a `generic' property. 
\begin{defn}\label{def:generic}
A property in terms of variables $a_1, a_2, \ldots, a_g$ is said to be satisfied generically if the set $P \subset \R^g$ of values that \underline{do}~\underline{not} satisfy the property is contained in the zero set of some non-zero polynomial in $a_1, a_2, \ldots, a_g$.
\end{defn}

Thus an uncontrollable realization of a structurally controllable system becomes controllable by arbitrarily small perturbations of some of its entries. Consequently, the set of all uncontrollable realizations $(A, B)$ of a structurally controllable system is closed and thin set \footnote{Algebraic variety is `thin' and set of measure zero.}. Thus, structural controllability gives us information about controllability of almost all network realizations without knowing exact numerical entries of its state matrix. 

Structural controllability can be tested in polynomial time \cite{Lin:74}. Here, we propose an alternate flow network based condition to establish structural controllability. Subsequently, for controllable systems, we develop algorithms based on this flow network to find optimal solutions to the optimization problems considered. First, we formally define the optimization problems. Let $(\bA, \bB)$ be structurally controllable.
Consider $\W \subseteq \{1,\ldots,m\}$ and let $\bB_\W$ be the restriction of $\bB$ to columns only in $\W$. Furthermore, let $\K = \{\W : (\bA,\bB_\W) \mbox{ is structurally controllable}\}$. The set $\K$ is non-empty, since for $\W =\{1,\ldots,m\}$, $(\bA, \bB_{\W}) = (\bA, \bB)$ is structurally controllable.

Now we formulate the two problems considered in this paper. Given a structurally controllable system $(\bA, \bB)$, the {\it minimum constrained input selection} (minCIS) problem consists of finding the least cardinality $\J \in \K$. Specifically, we wish to solve the following optimization:
\begin{prob}\label{prob:one}
 Given ($\bA, \bB)$, find
\[ \J^\*~ \in~ \arg\min_{\10n \J \in \K} |\J|. \] 
\end{prob}
Given a structurally controllable  structured system $(\bA$, $\bB)$ and non-negative cost vector $p_u$, where every entry $p_u(j)$, $j=1,2,\ldots,m$, indicates the cost of actuating $j^{\rm th}$ input, the {\it minimum cost constrained input selection} (minCCIS) problem consists of finding a minimum cost input set such that the system is structurally controllable. Specifically, we wish to solve the following optimization: For any $\I \in \K$, define $p(\I) = \sum_{j\in\I}p_u(j)$. 
\begin{prob}\label{prob:two}
Given structurally controllable $(\bA, \bB)$ and $p_u(j), j = 1,2,\ldots, m$, find
\[ \I^\* ~\in~ \arg\min_{\10n \I \in \K} p(\I). \] 
\end{prob}
Let $p^\* = p(\I^\*)$. Thus, $p^\*$ denotes the minimum cost
for constrained input selection that ensures structural controllability. Problem \ref{prob:one} is a special case of Problem \ref{prob:two} when all costs are non-zero and uniform, i.e., $p_u(j) = 1$, for every $j = 1,2,\ldots, m$.

The Problem \ref{prob:one} is shown to be NP-hard\footnote{NP-hard result for Problem \ref{prob:one} is obtained by reducing decision problem corresponding to a well known NP-hard problem, the minimum set covering problem, to an instance of decision verification of Problem \ref{prob:one}.} \cite{PeqSouPed:15}. Thus, its more general case, Problem \ref{prob:two}, is also NP-hard. A subclass of Problem \ref{prob:one} where the state bipartite graph (see Section \ref{sec:graph} for details) has a perfect matching and input matrix is diagonal is considered in \cite{PeqSouPed:15}. This problem is polynomially solvable and a polynomial solution is proposed in \cite{PeqSouPed:15}. Another subclass of Problem \ref{prob:two} where the graph is irreducible \footnote{a graph is said to be irreducible if there exist a directed path between any two distinct pair of nodes} is considered in \cite{PeqKarPap:15}. This subclass is also polynomially solvable and a polynomial solution is given in \cite{PeqKarPap:15}. Our aim is to deal with $(\bA, \bB)$ in their full generality. We formulate both these problems as instances of minimum cost fixed flow problem, where the objective is to minimize the cost associated with the flow. Since Problem \ref{prob:one} is a special case of Problem \ref{prob:two}, from now on we will discuss Problem \ref{prob:two} only. All the analysis and results directly applies to Problem \ref{prob:one}. Before discussing minimum cost flow formulation of Problem \ref{prob:two}, we explain structural controllability using concepts of graph theory in the next section.
\section{Graph Theoretic Results for Structural Controllability}\label{sec:graph}

 \begin{table}[t]
\begin{minipage}{8.5cm}
 \centering
\captionof{table}{Key Notations} 
\resizebox{\textwidth}{!}{%
\begin{tabular}{|c|}
\hline
{Notations\footnote{Digraph refers to directed graphs and graph refers to undirected graphs. Also, $E$ denotes directed edges and $\E$ denotes undirected edges.}}  \\
\hline	
\multicolumn{1}{|l|}{Set of states $V_{\x} = \{x_1,\ldots,x_n \}$} \\
\hline
\multicolumn{1}{|l|} {Set of inputs $V_{\u}$ = $\{u_1,\ldots,u_m \}$}\\
\hline
\multicolumn{1}{|l|} {Set of states $V_{\dx}$ = $\{x'_1,\ldots,x'_n \}$}\\
\hline
\multicolumn{1}{|l|} {Set of edges $E_{\x}$ = $\{(x_j, x_i):\bA_{ij} \neq 0 \}$}\\
\hline
\multicolumn{1}{|l|} {Set of edges $E_{\u}$ = $\{(u_j, x_i):\bB_{ij} \neq 0 \}$}\\
\hline
\multicolumn{1}{|l|} {State digraph $\D(\bA)$ = $\D(V_{\x}, E_{\x})$}\\
\hline
\multicolumn{1}{|l|} {System digraph $\D(\bA, \bB)$ = $\D(V_{\x}\cup V_{\u}, E_{\x}\cup E_{\u})$}\\
\hline
\multicolumn{1}{|l|} {State bipartite graph $\B(\bA)$ = $\B((V_{\x},V_{\dx}), \E_{\x})$}\\
\hline
\multicolumn{1}{|l|} {System bipartite graph $\B(\bA, \bB)$ = $\B((V_{\x}\cup V_{\u},V_{\dx}), (\E_{\x} \cup \E_{\u}))$}\\
\hline
\multicolumn{1}{|l|} {Maximum flow digraph = $\FAB$} \\
 \hline
\multicolumn{1}{|l|} {Minimum cost maximum flow digraph = $\Fc$}\\
\hline
\multicolumn{1}{|l|} {Maximum flow vector of $\FAB$ = $f^\*$}\\
\hline
\multicolumn{1}{|l|} {Optimal flow vector of MCFF on $\Fc$ = $f^\*_M$}\\
\hline
\multicolumn{1}{|l|} {Flow vector constructed in Algorithm~\ref{alg:twostage} = $f_A$}\\
\hline
\multicolumn{1}{|l|} {Optimal flow vector to Problem~\ref{prob:opt} = $f^\*_{LP}$}\\
\hline

\end{tabular} }
\label{tab:notations}
\end{minipage}
\end{table}

 In this section, we briefly describe some existing graph theoretic concepts associated with structural controllability for the sake of completeness (see \cite{LiuBar:16} for details). Key notations are summarized in Table \ref{tab:notations} for easy reference.
 
  The key idea behind considering graph for studying structural controllability is because we can represent the influences of states and inputs on each state through a directed graph. For a structured system $(\bA, \bB)$, construction of the system digraph involves two stages. In the first stage, effects of states on other states is captured and in the second stage effects of inputs on states is captured. In stage one of graph construction we
 construct the state digraph $\D(\bA) := \D(V_{\x}, E_{\x})$, where $V_{\x} = \{x_1, \ldots, x_n \}$ and $(x_j, x_i) \in E_{\x}$ if $\bA_{ij} \neq 0$. Presence of an edge $(x_j, x_i)$ in $\D(\bA)$ indicates that state $x_j$  of the system {\it influences} state $x_i$. 
 To capture the effect of inputs, we construct the system digraph $\D(\bA, \bB) := \D(V_{\x} \cup V_{\u}, E_{\x} \cup E_{\u})$. An edge $(u_j, x_i) \in E_{\u}$ if $\bB_{i, j} \neq 0$ and we say that input $u_j$ influences state $x_i$. Construction of state digraph $\D(\bA)$ and system digraph $\D(\bA, \bB)$ is illustrated through an example in Figure \ref{fig:eg}.

 A system is said to be controllable if it is possible to drive the system to any desired state by applying appropriate input. Thus, for a system to be controllable, it is essential that all states are influenced by inputs. A state gets influenced by input in two different ways, either directly or indirectly and then we say that the state is {\it accessible}. We will explain this concept of state accessibility in the context of structural controllability in more detail here. For a digraph a sequence of directed edges $\{(v_1, v_2), (v_2, v_3), \ldots, (v_{k-1}, v_{k})\}$ in which all vertices are distinct is called an {\it elementary path} from $v_1$ to $v_k$. If there exists an elementary path from vertex $v_1$ to vertex $v_k$, we say that $v_k$ is {\it reachable} from $v_1$. In a structured system, a state $x_i$ is said to be {\it inaccessible} if it is not reachable from any input vertex. Thus an inaccessible state cannot be influenced by any input and the system is uncontrollable. Figure \ref{fig:eg1} demonstrates the concept of state inaccessibility. The two subfigures of Figure \ref{fig:eg1} have the same $\bA$ but different $\bB$'s, say $\bB_1$ and $\bB_2$. The input structure in Figure \ref{fig:digraph1} is such that states $x_1$ and $x_3$ are inaccessible. This is because state $x_2$ gets directly influenced by input $u_1$, but states $x_1$ and $x_3$ are not reachable from $u_1$. However, in Figure \ref{fig:digraph2} the number of inputs is same as before, but all the states are accessible. Here, state $x_1$ gets influenced by input $u_1$ directly, while states $x_2$ and $x_3$ get influenced by $u_1$ indirectly through $x_1$.

\begin{figure}[t]
\centering
\begin{equation*} \label{eq:Amatrix}
\bA = 
\begin{bmatrix}
\* & \* & 0 & 0\\
0 & \* & 0 & 0\\
\* & \* & 0 & \* \\
0 & 0 & 0 & \*\\
\end{bmatrix}
\bB = 
\begin{bmatrix}
\* & 0 & \* \\
0 & \* & \* \\
\* & \* & 0  \\
0 & 0 & \* \\
\end{bmatrix}
\end{equation*}
\begin{subfigure}[b]{0.2\textwidth}
\begin{tikzpicture}[->,>=stealth',shorten >=0.5pt,auto,node distance=2cm,
                thick,main node/.style={circle,draw,font=\Large\bfseries}]

\tikzset{every loop/.style={min distance=10mm,looseness=10}}
    
\path[->] (-0.1,0.15) edge [in=60,out=100,loop] node[auto] {} ();
\path[->] (1.9,0.15) edge [in=60,out=100,loop] node[auto] {} ();
\path[->] (-1.0,-1.65) edge [in=-60,out=-100,loop] node[auto] {} ();
                
\definecolor{myblue}{RGB}{80,80,160}
\definecolor{mygreen}{RGB}{80,160,80}
\definecolor{myred}{RGB}{144, 12, 63}
\definecolor{myyellow}{RGB}{214, 137, 16}

  \node at (0.4,0.2) {\small $x_1$};
  \node at (-1,-1.1) {\small $x_2$};
  \node at (1,-1.1) {\small $x_3$};
  \node at (2.4,0.2) {\small $x_4$};
    
  \fill[myblue] (0,0) circle (5.0 pt);
  \fill[myblue] (2.0,0) circle (5.0 pt);
  \fill[myblue] (1.0,-1.5) circle (5.0 pt);
  \fill[myblue] (-1.0,-1.5) circle (5.0 pt);
  
  \draw (0,-0.15)  ->  (0.85,-1.5);
  \draw (-0.85,-1.5)  ->  (0.85,-1.5);
  \draw (-0.85,-1.5)  ->  (0,-0.15);
  \draw (2,-0.15)  ->  (1.15,-1.5);
     \end{tikzpicture}
\caption{$\D(\bA)$}
\label{fig:digraph}
\end{subfigure}~\hspace{3 mm}
\begin{subfigure}[b]{0.2\textwidth}
\begin{tikzpicture}[->,>=stealth',shorten >=0.5pt,auto,node distance=2cm,
                thick,main node/.style={circle,draw,font=\Large\bfseries}]

\tikzset{every loop/.style={min distance=10mm,looseness=10}}
    
\path[->] (-0.1,0.15) edge [in=60,out=100,loop] node[auto] {} ();
\path[->] (2,-0.15) edge [in=-60,out=-100,loop] node[auto] {} ();
\path[->] (-1.0,-1.65) edge [in=-60,out=-100,loop] node[auto] {} ();
                
\definecolor{myblue}{RGB}{80,80,160}
\definecolor{mygreen}{RGB}{80,160,80}
\definecolor{myred}{RGB}{144, 12, 63}
\definecolor{myyellow}{RGB}{214, 137, 16}

  \node at (-0.3,0.1) {\small $x_1$};
  \node at (-1,-1.1) {\small $x_2$};
  \node at (1,-1.8) {\small $x_3$};
  \node at (2.4,0.2) {\small $x_4$};
    
  \fill[myblue] (0,0) circle (5.0 pt);
  \fill[myblue] (2.0,0) circle (5.0 pt);
  \fill[myblue] (1.0,-1.5) circle (5.0 pt);
  \fill[myblue] (-1.0,-1.5) circle (5.0 pt);
  
  \draw (0,-0.15)  ->  (0.85,-1.5);
  \draw (-0.85,-1.5)  ->  (0.85,-1.5);
  \draw (-0.85,-1.5)  ->  (0,-0.15);
  \draw (2,-0.15)  ->  (1.15,-1.5);

  \node at (1.4,1) {\small $u_1$};
  \node at (0.4,-3) {\small $u_2$};
  \node at (2.4,1) {\small $u_3$};

  \fill[myred] (1,1) circle (5.0 pt);
  \draw [myred] (1,1)  ->   (0.15,0);
  \draw [myred] (1,1)  ->   (1,-1.35);
  \fill[myred] (0,-3) circle (5.0 pt);
  \draw [myred] (0,-3)  ->   (0.85,-1.5);
  \draw [myred] (0,-3)  ->   (-0.85,-1.5);  
  \fill[myred] (2,1) circle (5.0 pt);
  \draw [myred] (2,1)  ->   (2,0.15);
  \draw [myred] (2,1)  ->   (0.15,0);
  \draw [myred] (2,1)  ->   (-0.85,-1.5);        
\end{tikzpicture}
\caption{$\D(\bA, \bB)$}
\label{fig:digraphb}
\end{subfigure}~\hspace{-4 mm}
\caption{The state digraph and system digraph representations of the structured system $(\bA, \bB)$ are shown in Figure \ref{fig:digraph} and Figure \ref{fig:digraphb} respectively.}
\label{fig:eg}
\end{figure}
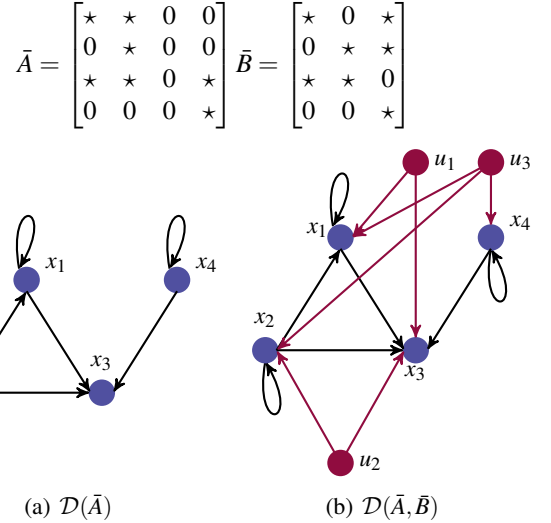

An alternate method for checking if all states are accessible is by using a concept of strong connectedness of the graph. A digraph is said to be strongly connected if for each ordered pair of vertices $(v_1,v_k)$
there exists an elementary path from $v_1$ to $v_k$. A digraph may not always be strongly connected. In such a case, we look at the maximal strongly connected subgraphs of it. A maximal strongly connected subgraph of a digraph, called a {\it strongly connected component} (SCC),  is a subgraph that is strongly connected and is not properly contained in any other subgraph that is strongly connected. To check if all states are accessible, we first generate a {\it directed acyclic graph} (DAG) associated with $\D(\bA)$ by condensing each SCC to a supernode. Thus in this DAG, vertex set comprises of all SCC's. A directed edge exists between two nodes of DAG {\it if and only if} there exists a directed edge connecting vertices in the respective SCC's in the original digraph. Using this DAG, we have a following definition characterizing SCC's of $\D(\bA)$.

\begin{defn}\label{def:scc}
An SCC is said to be \underline{linked} if it has atleast one incoming or outgoing edge from another SCC. Further, an SCC is said to be \underline{non-top} \underline{linked} if it has no incoming edges to its vertices from the vertices of another SCC.
\end{defn}

Clearly, a digraph has no inaccessible states if and only if all \mbox{non-top} linked SCC's are connected to some input vertex. In the example given in Figure \ref{fig:eg}, there are four SCC's, $\{ \{x_1\}, \{x_2\}, \{x_3\}, \{x_4\}\}$. However, there are two \mbox{non-top} linked SCC's, $\cN_1 = \{x_2\}, \cN_2 = \{x_4 \}$.

\begin{figure}[t]
\centering
\begin{subfigure}[b]{0.2\textwidth}
\begin{tikzpicture}[->,>=stealth',shorten >=0.5pt,auto,node distance=2cm,
                thick,main node/.style={circle,draw,font=\Large\bfseries}]
    
\definecolor{myblue}{RGB}{80,80,160}
\definecolor{mygreen}{RGB}{80,160,80}
\definecolor{myred}{RGB}{144, 12, 63}
\definecolor{myyellow}{RGB}{214, 137, 16}

  \node at (-0.0,0.3) {\small $x_1$};
  \node at (-1,-1.8) {\small $x_2$};
  \node at (1,-1.8) {\small $x_3$};

  \fill[myblue] (0,0) circle (5.0 pt);
  \fill[myblue] (1.0,-1.5) circle (5.0 pt);
  \fill[myblue] (-1.0,-1.5) circle (5.0 pt);
  
  \draw (0,-0.15)  ->  (0.85,-1.5);
  \draw (0,-0.15)  ->  (-0.85,-1.5);

  \node at (-1,-0.2) {\small $u_1$};

  \fill[myred] (-1,-0.5) circle (5.0 pt);
  \draw [myred] (-1,-0.5)  ->   (-1,-1.4);
     
\end{tikzpicture}
\caption{$\D(\bA, \bB_1)$}
\label{fig:digraph1}
\end{subfigure}~\hspace{3 mm}
\begin{subfigure}[b]{0.2\textwidth}
\begin{tikzpicture}[->,>=stealth',shorten >=0.5pt,auto,node distance=2cm,
                thick,main node/.style={circle,draw,font=\Large\bfseries}]

\definecolor{myblue}{RGB}{80,80,160}
\definecolor{mygreen}{RGB}{80,160,80}
\definecolor{myred}{RGB}{144, 12, 63}
\definecolor{myyellow}{RGB}{214, 137, 16}

  \node at (-0.0,0.3) {\small $x_1$};
  \node at (-1,-1.8) {\small $x_2$};
  \node at (1,-1.8) {\small $x_3$};

  \fill[myblue] (0,0) circle (5.0 pt);
  \fill[myblue] (1.0,-1.5) circle (5.0 pt);
  \fill[myblue] (-1.0,-1.5) circle (5.0 pt);
  
  \draw (0,-0.15)  ->  (0.85,-1.5);
  \draw (0,-0.15)  ->  (-0.85,-1.5);

  \node at (0,-1.35) {\small $u_1$};

  \fill[myred] (0,-1) circle (5.0 pt);
  \draw [myred] (0,-1)  ->   (0,-0.18);
     
\end{tikzpicture}
\caption{$\D(\bA, \bB_2)$}
\label{fig:digraph2}
\end{subfigure}~\hspace{-4 mm}
\caption{Example demonstrating inaccessibility and dilation of nodes.}
\label{fig:eg1}
\end{figure}
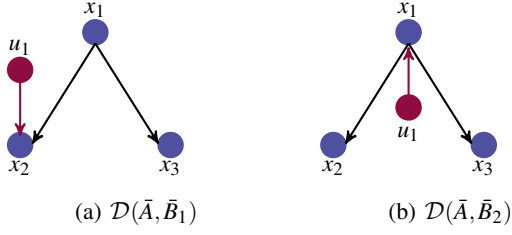

While accessibility of all states is necessary for structural controllability, it is not sufficient. For example see Figure \ref{fig:digraph2}. Note that all states are accessible in this case. However, the system is not controllable. This is because of the fact that the state $x_1$ alone has to control both states $x_2$ and $x_3$. Thus it is not possible to control the difference between these states independently. Thus all states being accessible is not enough to guarantee controllability. In addition to this, we must also ensure that given a set of nodes $S \subset V_{\x}$, the neighbourhood node set of $S$, denoted by $T(S)$ (where node $x_i \in T(S)$, if there exists a directed edge from $x_i$ to a node in $S$) does not have fewer nodes than $S$. Note that, $S \subset V_{\x}$ but $T(S) \subset V_{\x} \cup V_{\u}$. Presence of $S$ such that $|T(S)| < |S|$ is called as {\it dilation}. In short, the accessibility of states and the absence of dilations are necessary for structural controllability. Formally, Lin proved the sufficiency of these two conditions through the following result. 

\begin{prop}[\cite{Lin:74}]\label{prop:lin}
The structured system $(\bA,\bB)$ is structurally controllable if and only if the associated digraph $\D(\bA,\bB)$ has no inaccessible states and has no dilations.
\end{prop}

Checking for dilation in a digraph by brute-force technique is computationally intensive. However, there is a necessary and sufficient condition for determining existence of dilation in terms of matchings of a bipartite graph \cite{Ols:15}. In a bipartite graph $G((V, \widetilde{V}), E)$, where $V \cup \widetilde{V}$ denote the set of nodes satisfying $V \cap \widetilde{V} = \phi$ and $E \subseteq V \times \widetilde{V}$ denote set of undirected edges, a matching $M$ is a collection of edges $M \subseteq E$ such that no two edges in the collection share the same endpoint. That is, for any $(i, j)$ and $(u, v) \in M$, we have $i \neq u$ and $j \neq v$, where $i,u \in V$ and $j,v \in \widetilde{V}$. For understanding the relation between bipartite matching and absence of dilation, we first explain how a bipartite graph is constructed from the system matrices. Corresponding to the state digraph $\D(\bA)= \D(V_{\x},E_{\x})$, we associate the bipartite graph $\B(\bA) := \B((V_{\x},V_{\dx}), \E_{\x})$, where $V_{\x}=\{x_1,x_2, \ldots, x_n \}$, $V_{\dx}=\{x'_1,x'_2, \ldots, x'_n \}$ and $(x_i, x'_{j}) \in \E_{\x} \Leftrightarrow (x_i, x_j) \in E_{\x}$. Similarly, corresponding to the system digraph $\D(\bA, \bB) = \D(V_{\x} \cup V_{\u},E_{\x} \cup E_{\u})$, we associate the bipartite graph $\B(\bA, \bB) := \B((V_{\x} \cup V_{\u}, V_{\dx}), (\E_{\x} 
 \cup \E_{\u}))$, where $V_{\u} = \{u_1,u_2, \ldots, u_m \}$ and $(u_i, x'_{j}) \in \E_{\u} \Leftrightarrow (u_i, x_j) \in E_{\u}$. Note that $\D(\bA), \D(\bA,\bB)$ are digraphs, but $\B(\bA), \B(\bA, \bB)$ are undirected graphs. The bipartite representation of the system given in Figure \ref{fig:eg} is shown in Figure \ref{fig:eg2}. If there exists a perfect matching in $\B(\bA, \bB)$, then it is clear that for any $S \subset V_{\x}$, $|T(S)| \geqslant |S|$. As a result, existence of a perfect matching in $\B(\bA, \bB)$ implies absence of dilation in the digraph $\D(\bA, \bB)$. Further, if there is no dilation, then there exists a perfect matching. This can be understood using a contradiction argument. If there is no dilation and no perfect matching, then consider a maximum matching $M$ in $\B(\bA, \bB)$. Let $U' \subset V_{\dx}$ denote the set of matched vertices in $M$ and let vertex $x'_k$ is unmatched in the matching $M$. Then define $U$ as the set 
of vertices such that $x_i \in U \Leftrightarrow x'_i \in U'$. Then for 
$S = U \cup x_k$, $|S| = |U|+1$. However, $|T(S)| = |U|$ and hence $S = U \cup x_k$ is a dilation. Thus, if there exists no perfect matching in $\B(\bA, \bB)$, then there exists a dilation in $\D(\bA, \bB)$. 
\begin{prop}[\cite{Ols:15}, Theorem 2]\label{prop:dil}
 A digraph $\D(\bA,\bB)$ has no dilations if and only if the bipartite graph $\B(\bA,\bB)$ has a perfect matching.
\end{prop}

\begin{figure}
\begin{subfigure}[b]{0.2\textwidth}
\centering
\definecolor{myblue}{RGB}{80,80,160}
\definecolor{mygreen}{RGB}{80,160,80}
\definecolor{myred}{RGB}{144, 12, 63}
\begin{tikzpicture} [scale = 0.2]
       \draw [thick] (-5,-5)  --   (5,-5);
       \draw [thick] (-5,-5)  --   (5,-10);
       
       \draw [thick] (-5,-7.5)  --   (5,-5);
       \draw [thick] (-5,-7.5)   --   (5,-7.5);
       \draw [thick] (-5,-7.5)  --   (5,-10);

       \draw [thick] (-5,-12.5)   --   (5,-10);
       \draw [thick] (-5,-12.5)  --   (5,-12.5);       
            
          \node at (-7,-5) {\small $x_1$};
          \node at (-7,-7.5) {\small $x_2$};
          \node at (-7,-10.0) {\small $x_3$};
          \node at (-7,-12.5) {\small $x_4$};
          \fill[myblue] (-5,-5) circle (30.0 pt);
          \fill[myblue] (-5,-7.5) circle (30.0 pt);
          \fill[myblue] (-5,-10) circle (30.0 pt);
          \fill[myblue] (-5,-12.5) circle (30.0 pt);
          \node at (7,-5) {\small $x'_1$};
          \node at (7,-7.5) {\small $x'_2$};
          \node at (7,-10) {\small $x'_3$};
          \node at (7,-12.5) {\small $x'_4$};
         
          \fill[mygreen] (5,-5) circle (30.0 pt); 
          \fill[mygreen] (5,-7.5) circle (30.0 pt);
          \fill[mygreen] (5,-10) circle (30.0 pt);
          \fill[mygreen] (5,-12.5) circle (30.0 pt);
          \node at (-5.0,-3) {$V_{\x}$};
          \node at (5,-3) {$V_{\dx}$};
\end{tikzpicture}
\caption{$\B(\bA)$}
\label{fig:bip1}
\end{subfigure}~\hspace{0 mm}
\begin{subfigure}[b]{0.2\textwidth}
\centering
\definecolor{myblue}{RGB}{80,80,160}
\definecolor{mygreen}{RGB}{80,160,80}
\definecolor{myred}{RGB}{144, 12, 63}
\begin{tikzpicture} [scale = 0.2]
       \draw [dashed] (-5,-5)  --   (5,-5);
       \draw [dashed] (-5,-5)  --   (5,-10);
       
       \draw [dashed] (-5,-7.5)  --   (5,-5);
       \draw [dashed] (-5,-7.5)   --   (5,-7.5);
       \draw [dashed] (-5,-7.5)  --   (5,-10);

       \draw [dashed] (-5,-12.5)   --   (5,-10);
       \draw [dashed] (-5,-12.5)  --   (5,-12.5);
       
       \draw [thick, myred] (-5,-15)  --   (5,-5);
       \draw [thick, myred] (-5,-15)  --   (5,-10);
       
       \draw [thick,myred] (-5,-17.5)  --   (5,-10);
       \draw [thick,myred] (-5,-17.5)   --   (5,-7.5);

       \draw [thick, myred] (-5,-20)   --   (5,-5);
       \draw [thick, myred] (-5,-20)   --   (5,-7.5);
       \draw [thick, myred] (-5,-20)  --   (5,-12.5);     
          
          \node at (-7,-5) {\small $x_1$};
          \node at (-7,-7.5) {\small $x_2$};
          \node at (-7,-10.0) {\small $x_3$};
          \node at (-7,-12.5) {\small $x_4$};
          \node at (-7,-15) {\small $u_1$};
          \node at (-7,-17.5) {\small $u_2$};
          \node at (-7,-20) {\small $u_3$};
          \fill[myblue] (-5,-5) circle (30.0 pt);
          \fill[myblue] (-5,-7.5) circle (30.0 pt);
          \fill[myblue] (-5,-10) circle (30.0 pt);
          \fill[myblue] (-5,-12.5) circle (30.0 pt);
          \fill[myblue] (-5,-15) circle (30.0 pt);
          \fill[myblue] (-5,-17.5) circle (30.0 pt);
          \fill[myblue] (-5,-20) circle (30.0 pt);
          \node at (7,-5) {\small $x'_1$};
          \node at (7,-7.5) {\small $x'_2$};
          \node at (7,-10) {\small $x'_3$};
          \node at (7,-12.5) {\small $x'_4$};
         
          \fill[mygreen] (5,-5) circle (30.0 pt); 
          \fill[mygreen] (5,-7.5) circle (30.0 pt);
          \fill[mygreen] (5,-10) circle (30.0 pt);
          \fill[mygreen] (5,-12.5) circle (30.0 pt);
          \node at (-5.0,-3) {$V_{\x} \cup V_{\u}$};
          \node at (5,-3) {$V_{\dx}$};
\end{tikzpicture}
\caption{$\B(\bA,\bB)$}
\label{fig:bip2}
\end{subfigure}
\caption{The bipartite graph representation of the structured system $(\bA,\bB)$ given in Figure \ref{fig:eg}.}
\label{fig:eg2}
\end{figure}
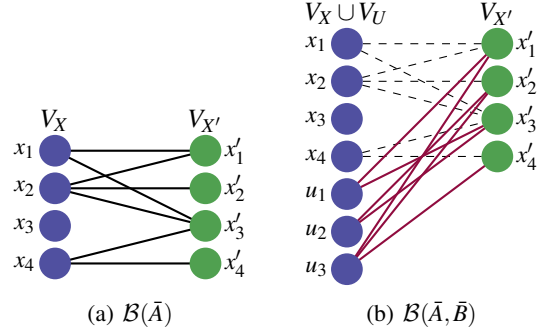

Therefore, a structured system is said to be controllable, if and only if all non-top linked SCC's get influenced by some input and there exists a perfect matching in the bipartite graph $\B(\bA, \bB)$. 
 Finding the non-top linked SCC's involve $O(n^2)$ computations and checking for existence of perfect matching involve $O(n^{2.5})$ computations \cite{CorLeiRivSte:01}. Thus structural controllability of a system can be accurately checked in $O(n^{2.5})$ computations. Using the two graph theoretical conditions explained in this section, we conclude that checking structural controllability of a system has polynomial complexity. However, these conditions do not give ample insight about solving minCCIS problem. In the next section, we give an alternate graph theoretical condition for checking structural controllability using flow networks. This condition will be subsequently used to provide approximation algorithms for minCCIS.

\section{Structural Controllability and Maximum Flow Problem}\label{sec:equi}
In this section, we establish a relation between the structural controllability and the maximum flow problem \cite{ForFul:55}. The maximum flow problem is a classical problem, where the objective is to find the maximum flow through a single source-sink flow network under certain capacity constraint. Here, given a flow network $\F$ with vertex set $V$, directed edge set $E$, source-sink pair $s,t$ and non-negative capacities $b(e)$ for every $e \in E$, we
define a flow vector $f$ as a function from the edge set $E$ to the set of non-negative real numbers $\R_+$.

\begin{defn}\label{def:feasible_flow}
In a flow network $\F$ with vertex and edge sets $V$ and $E$ respectively, a source-sink pair $s,t$ and non-negative edge capacities $b(e)$, a flow vector $f$ is said
to be feasible if (a)~$f(e) \leqslant b(e)$ for every $e \in E$, and (b)~$\sum_{ e=(u,v)\in E}f(e) = \sum_{ e'= (v,u)\in E}f(e')$ for every $v \in V \setminus \{ s, t\}$.
\end{defn}

The requirement (a) ((b), resp.) in Definition~\ref{def:feasible_flow} is called capacity
(flow conservation, resp.) constraint. Capacity constraint ensures that the flow through each edge is less than the edge capacity. The flow conservation constraint ensures that at every node, except the source and sink nodes, the flow leaving the node equals the flow entering the node. We define the flow from the source to the sink under a feasible flow vector $f$ as
\begin{align} \label{eq:flow}
\varphi_f = \sum_{e = (s,v) \in E}f(e).
\end{align}
The objective of a maximum flow problem is to find a feasible flow vector $f^\*$ such that $\varphi_{f^\*} \geqslant \varphi_f$
for any feasible flow vector $f$. It is a well studied problem and there exist many algorithms that find the maximum flow $f^\*$ in time polynomial in the number of nodes and edges of the flow network. For example, the algorithm in \cite{Orl:13} computes maximum flow in $O(|V||E|)$.

\begin{algorithm}[t]
  \caption{Pseudo code for constructing flow network of a structured system $(\bA, \bB)$
  \label{alg:SC to mfp}}
  \begin{algorithmic}
\State \textit {\bf Input:} Structured matrices $\bA \in \{0,\* \}^{n \times n}$ and $\bB \in \{ 0,\* \}^{n \times m}$
\end{algorithmic}
  \begin{algorithmic}[1]
  \State  Find non-top linked SCC's $\cN = \{\cN_i\}_{i=1}^q$ \label{step:find}
  \State Construct flow network $\FAB$ with vertex set $V_F$ and edge set $E_F$ as follows:  \label{step:graph}
  \State  $V_F \leftarrow \Big \{ \{s,t\} \cup \{\cN_i\}_{i=1}^q \cup \{x'_k\}_{k=1}^n \cup \{x_r\}_{r=1}^n \cup \{u_j\}_{j=1}^m \cup \{u'_j\}_{j=1}^m \Big \}$\label{step:vertex}
  \State $e \in E_F \leftarrow \begin{cases}
(s, \cN_i),~ {\rm for}~ i \in \{ 1,2,\ldots, q \},\\
(s, x'_k), ~ {\rm for}~ k \in \{ 1,2,\ldots, n\},\\
(\cN_i, u'_j), ~ \bB_{r,j} = \*~  \mbox{and}~ x_r \in \cN_i, \\
(x'_k, x_r),~\bA_{kr} \neq 0, \\
(x'_k, u_j), ~\bB_{kj} \neq 0,\\
(u_j, u'_j), ~ {\rm for}~j \in \{1,2,\ldots,m \},\\
(u'_j, t), ~ {\rm for}~ j \in \{1,2,\ldots,m \},\\
(x_r, t), ~ {\rm for}~ r \in \{1,2,\ldots,n \}.
\end{cases}$\label{step:edge}
\State $b(e) \leftarrow \begin{cases}
n + 1 , ~~ {\rm for}~ e = (u'_j,t), j \in \{1,2,\ldots,m\},\\
1,~~ \mbox{otherwise.} 
\end{cases}$ \label{step:capa}
\end{algorithmic}
\begin{algorithmic}
\State \textit{\bf Output:} Flow network $\FAB, s,t, b$
  \end{algorithmic}
\end{algorithm}
In order to establish a relation between these two problems, i.e., maximum flow and structural controllability, we first construct the flow network $\FAB$ corresponding to the given structured system $(\bA, \bB)$. The pseudo code for constructing the flow network $\FAB$ is presented in Algorithm \ref{alg:SC to mfp}. Given $(\bA, \bB)$, we first find the digraph $\D(\bA)$, the bipartite graph $\B(\bA, \bB)$ and the non-top linked SCC's in $\D(\bA)$, $\cN = \{ \cN_i \}_{i = 1}^q$ (see Step \ref{step:find}). Then we define the vertex set $V_F$ (see Step \ref{step:vertex}), edge set $E_F$ (see Step \ref{step:edge}), source-sink pair $s,t$ and capacity vector $b$ (see Step \ref{step:capa}) as shown in the algorithm\footnote{Note that even though $V_F$ and $E_F$ depend on $(\bA,\bB)$, we are not making the dependence explicit in our notations for brevity. We believe that $(V_F,E_F)$ can be obtained unambiguously given the context.}. The flow network $\FAB$ of system $(\bA, \bB)$ given in Figure \ref{fig:eg} is shown in Figure \ref{fig:flowgraph}. Note that in $\FAB$, Block-1 corresponds to the non-top linked SCC's in $\D(\bA)$ and Block-2 is the directed version of $\B(\bA, \bB)$ shown in Figure \ref{fig:bip2}. The flows entering Block-1 and Block-2 are defined as $\sum_{e \in \{(s, \cN_i)\}_{i=1}^q}f(e)$ and $\sum_{e \in \{(s, x'_k)\}_{k=1}^n}f(e)$ respectively. These flows are critically used in the sequel to ensure that no state is inaccessible and there are no dilations if the system $(\bA, \bB)$ is structurally controllable. In order to relate structural controllability and maximum flow problem we prove the following result.

\begin{figure}
\centering
\definecolor{myblue}{RGB}{80,80,160}
\definecolor{mygreen}{RGB}{80,160,80}
\definecolor{myred}{RGB}{144, 12, 63}
\definecolor{myyellow}{RGB}{214, 137, 16}
\begin{tikzpicture}[->,>=stealth',scale = 0.2]
   
       \draw [] (-16,-17.5)  ->   (-8,-7.5);
       \draw [] (-16,-17.5)  ->   (-7.8,-10.5);

       \draw [] (-7,-7.5)  ->   (9,-7.5);
       \draw [] (-7,-7.5)  ->   (9,-10);
       \draw [] (-7,-10.0)  ->   (9, -10);
       
       \draw [] (10,-3.5)   ->   (17, -20);
       \draw [] (10,-7.5)  ->   (17, -20);
       \draw [] (10,-10.0)  ->   (17, -20);
       
       \draw [] (3,-15)  ->   (9, -4.5);
       \draw [] (3,-17.5)  ->   (9,-7.5);
       \draw [] (3,-20)   ->   (9, -10);
       
       
          
          \node at (-18,-17.5) {\small $s$};
          \node at (20,-20) {\small $t$};
          \node at (-9,-7.0) {\small $\cN_1$};
          \node at (-9,-10.0) {\small $\cN_2$};
          \node at (12,-4.5) {\small $u'_1$};
          \node at (12,-7.5) {\small $u'_2$};
          \node at (12,-10.5) {\small $u'_3$};
          \fill[myred] (-16,-17.5) circle (30.0 pt);
          \fill[myred] (18,-20) circle (30.0 pt);
          \fill[mygreen] (-7,-7.5) circle (30.0 pt);
          \fill[mygreen] (-7,-10) circle (30.0 pt);
          \fill[myblue] (10,-4.5) circle (30.0 pt);
          \fill[myblue] (10,-7.5) circle (30.0 pt);
          \fill[myblue] (10,-10.5) circle (30.0 pt);           
          \node at (4,-15) {\small $u_1$};
          \node at (4,-17.5) {\small $u_2$};
          \node at (4,-20) {\small $u_3$};

          \fill[myblue] (2,-15) circle (30.0 pt);
          \fill[myblue] (2,-17.5) circle (30.0 pt);
          \fill[myblue] (2,-20) circle (30.0 pt);
 

       \draw [] (-15.5,-18.3)  ->   (-8,-22.5);
       \draw [] (-15.5,-18.3)  ->   (-8,-25);
       \draw [] (-15.5,-18.3)  ->   (-8,-27.5);
       \draw [] (-15.5,-18.3)  ->   (-8,-30);
       
       \draw [] (-7,-22.5)   ->   (1,-15);
       \draw [] (-7,-27.5)  ->   (1,-15);
       \draw [] (-7,-25)   ->   (1,-17.5);
       \draw [] (-7,-27.5)  ->   (1,-17.5); 
       \draw [] (-7,-25)   ->   (1,-20);
       \draw [] (-7,-30)  ->   (1,-20);

       \draw [] (-7,-22.5)  ->   (1,-22.5);
       \draw [] (-7,-22.5)  ->   (1, -25);
       
       \draw [] (-7,-25.0)  ->   (1,-25);
       \draw [] (-7,-27.5)  ->   (1, -22.5);
       \draw [] (-7,-27.5)   ->   (1, -25);
       \draw [] (-7,-27.5)  ->   (1, -30);
       
       \draw [] (-7,-22.5)  ->   (1, -20);
       \draw [] (-7,-30)   ->   (1, -30);
       
       \draw [] (3,-22.5)  ->   (17,-20);
       \draw [] (3,-25)  ->   (17,-20);
       \draw [] (3,-27.5)  ->   (17,-20);
       \draw [] (3,-30)  ->   (17,-20);     
          
          
          \node at (-9,-22.5) {\small $x'_1$};
          \node at (-9,-25) {\small $x'_2$};
          \node at (-9,-27.5) {\small $x'_3$};
          \node at (-9,-30) {\small $x'_4$};
          
          
          \fill[mygreen] (-7,-22.5) circle (30.0 pt);
          \fill[mygreen] (-7,-25) circle (30.0 pt);
          \fill[mygreen] (-7,-27.5) circle (30.0 pt);
          \fill[mygreen] (-7,-30) circle (30.0 pt);
          \node at (4,-22.5) {\small $x_1$};
          \node at (4,-25) {\small $x_2$};
          \node at (4,-27.5) {\small $x_3$};
          \node at (4,-30) {\small $x_4$};
          \fill[myblue] (2,-22.5) circle (30.0 pt);
          \fill[myblue] (2,-25) circle (30.0 pt);
          \fill[myblue] (2,-27.5) circle (30.0 pt); 
          \fill[myblue] (2,-30) circle (30.0 pt);

          
\draw[dashed] (-13,-12) rectangle (15,3);
\node at (10,1) {Block-1};
\node (rect) at (0.9,-22) [draw,dashed,minimum width=5.65cm,minimum height=3.75cm] {};
\node at (10,-30) {Block-2};
\end{tikzpicture} 
\caption{The flow network $\FAB$ for the structured system $(\bA,\bB)$ given in Figure \ref{fig:eg}.}
\label{fig:flowgraph}
\end{figure}
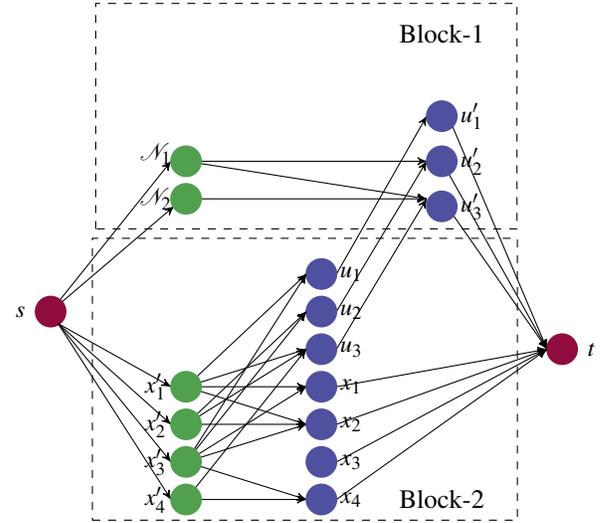

\begin{theorem}\label{th:maxflow}
Consider a structured system $(\bA, \bB)$. Let $n$ denote the number of states in the system and $q$ denote the number of non-top linked SCC's in $\D(\bA)$. Then, $(\bA, \bB)$ is structurally controllable if and only if the maximum flow $\varphi_{f^\*}$ in the flow network $\FAB$ is atleast $q+n$.
\end{theorem}
\begin{proof}
Recall integrality theorem in maximum flow which states that if all capacities in a flow network are integers, then there exists an integer maximum flow solution \cite{ForFul:55}. Since $b(e)~\in~\mathbb{Z}_+$, where $\mathbb{Z}$ is the set of all integers, for all $e \in E_F$, without loss of generality we assume that the optimal flow vector $f^\*$  
is an integer valued function from $E_F$. 
We will use this in proving both the if and only if parts of the theorem.

{\bf If part:} We prove that if $(\bA, \bB)$ is structurally controllable, then the maximum flow is atleast $q+ n$, i.e., $\varphi_{f^\*} \geqslant q+n$. Assume $(\bA, \bB)$ is structurally controllable. Then by Proposition \ref{prop:lin} and Proposition \ref{prop:dil}, all the states are accessible and there exists a perfect matching in $\B(\bA, \bB)$. All states being accessible implies that all \mbox{non-top} linked SCC's are connected to some input vertex. Denote by $u({\cN_i})$ an input that connects to some state in a non-top linked SCC $\cN_i$. There can be many inputs connecting to a vertex in $\cN_i$, we can choose anyone of them as $u(\cN_i)$.  Furthermore, since  $\B(\bA, \bB)$ has a perfect matching, say $M$, for every vertex $x'_k$ there exist a unique $y_k \in V_X \cup V_U$ such that $(x'_k,y_k) \in M$. The uniqueness of $y_k$'s ensure that $y_{k_1} = y_{k_2}$ only if
$k_1 = k_2$. Now, we construct a feasible flow vector $f$ for the flow network $\F(V_F,E_F)$ such that $\varphi_f = q + n$. This proves the required as $\varphi_{f^\*} \geqslant \varphi_f$. 

Construct a flow vector $f$ in $\FAB$ as follows:

\noindent
1.~$f((s,v)) = 1$ for every $v \in \{\cN_1,\ldots,\cN_q\}\cup\{x'_1,\ldots,x'_n\}$,

\noindent
2.~$f((\cN_i,u'({\cN_i}))) = 1$ for every $i \in \{1,\ldots,q\}$,

\noindent
3.~$f((x'_k,y_k))=1$, for every $k\in \{1,\ldots,n\}$,

\noindent
4.~if $y_i \in V_U$, then $f((y_i,y'_i)) = 1$, else $f(y_i,t) = 1$, and

\noindent
5.~$f((u'_j,t)) = |\{i:u_j = u(\cN_i)\}| + f((u_j,u'_j))$, where $|D|$ denotes the cardinality of set $D$.

From the construction step~1 and equation \eqref{eq:flow}, it follows that $\varphi_f = q+n$. Thus, it suffices to show
that the flow vector $f$ is feasible. First, we show that $f$ satisfies capacity constraint. Note from the construction steps~1 to~4, each edge except that emanating from nodes $u'_j$'s have unit flow. Thus, for these edges, we need to argue that they belong to $E_F$. 
Recall Algorithm~\ref{alg:SC to mfp}. Clearly, the edges in construction step~1 are in $E_F$. As defined above $u(\cN_i)$ denotes an input that connects to a vertex in $\cN_i$. Thus, edges $(\cN_i, u'(\cN_i))$ considered in construction step~2 are in $E_F$.
Recall that $y_k$ is defined so that $(x_k,y_k)$ is an edge in a perfect matching of $\B(\bA,\bB)$. Thus, the edges considered in step~3 and~4 also belong in $E_F$.

Now, we show that capacity constraint is satisfied. Note that the capacity $b(e)$ is at least one for every $e\in E_F$. Thus, if suffices to show that the capacity constraint is satisfied for the edges $(u'_j,t)$ as flow through all other edges is at most one.  
Edges $(u'_j,t)$ considered in construction step~5 are shown to be in $E_F$. By construction of $\FAB$, each of these edges has capacity $n+1$.
Since $q \leqslant n$ and $u'_j$ can have unit capacity incoming edges only from $u_j$ and $\cN_1,\ldots,\cN_q$, thus 
the total flow coming in $u'_j$ is bounded above by $n+1$. This concludes that $f$ satisfies capacity constraint.  

To see that $f$ satisfies flow conservation constraint, note that the flow being pushed in construction step~1 is pushed out in construction steps~2 and~3, subsequently this flow is further pushed to the sink $t$ in steps~4 and~5. Thus, $f$ is feasible, proving the required.

{\bf Only if part:} Here, we show that if $\varphi_{f^\*} \geqslant q+n$, then the system $(\bA,\bB)$ is structurally controllable. To establish the required, we show that when $\varphi_{f^\*} \geqslant q+n$, then both accessibility and no dilation conditions required for structural controllability are satisfied (recall Proposition~\ref{prop:lin}).

Let us assume $\varphi_{f^\*} \geqslant q+n$ in the flow network $\FAB$. Since there are exactly $q+n$ edges, each with capacity one, emanating from the source vertex $s$,
each of these edges should carry one unit flow. Since $f^\*$ is a feasible flow vector, it satisfies flow conservation at each node in $V_F \setminus \{s,t\}$. Specifically, the flow conservation is satisfied at nodes $\cN_1,\ldots,\cN_q$ and
$x'_1,\ldots,x'_n$. Thus, for every $\cN_i$, there exist
$u'_j$ such that $(\cN_i,u'_j)\in E_F$. By construction of
the flow network $\FAB$,  $(\cN_i,u'_j)\in E_F$ if the input $u_j$ connects to some state $x \in \cN_i$. Thus, all non-top linked SCC's are connected to atleast one input. This ensures that all states are accessible.

Furthermore, on account of flow conservation at nodes $x'_1,\ldots,x'_n$ and flow  integrality, there exists $y_k \in V_X \cup V_U$ such that $f^\*((x'_k,y_k))=1$. Since the capacity of outgoing edges from each node in $V_X \cup V_U$ is one, it follows that $y_{k_1} \not= y_{k_2}$ unless $k_1 = k_2$. Now, note that the set
$\{(x'_k,y_k):k=1,\ldots,n\}$ is a matching in $\B(\bA,\bB)$.
This proves the required using Propositions~\ref{prop:lin} and~\ref{prop:dil}. This completes the proof.
\end{proof}

Following result is an immediate consequence of Theorem~\ref{th:maxflow}.

\begin{cor}\label{cor:subset_sel}
Consider $\FAB$ and any feasible flow vector $f$ such that
$\varphi_f \geqslant q+n$ and define 
$\W_f = \{j: f(u'_j,t) >0\}$. Then, the structured system
$(\bA,\bB_{\W_f})$ is structurally controllable. 
\end{cor}

\begin{proof}
The result follows from Theorem~\ref{th:maxflow} and an observation that the maximum flow through $\F(\bA,\bB_{\W_f})$ is $q+n$.
\end{proof}

\begin{rem}
The above result allows for obtaining a subset of all possible inputs that are enough to retain controllability of the structured system from the obtained flow vector. 
Conversely, the structural controllability of the system
with a given subset of inputs, say $\W$, can be checked using Theorem~\ref{th:maxflow} for flow network $\F(\bA,\bB_{\W})$. \end{rem}   

In the following lemma, we show that $\FAB$ can be constructed in polynomial time. 
\begin{lem}\label{lem:flowgraph}
Constructing the flow network $\FAB$ corresponding to a structured system $(\bA, \bB)$ has complexity $O(n^2)$.
\end{lem}
\begin{proof}
Given the state digraph $\D(\bA) = \D(V_{\x}, E_{\x})$, the non-top linked SCC's can be found in $O(|V_{\x}| + |E_{\x}|)$ computations and the rest of the constructions in Algorithm \ref{alg:SC to mfp} are of linear complexity. Here, $|V_{\x}| = n$ and $|E_{\x}|$ is atmost $|V_{\x}|^2$. However, in large systems the state matrix is sparse and hence the number of edges in the state digraph is much less than the above bound. Thus the construction of the flow network $\FAB$ has complexity $O(n^2)$. This completes the proof.
\end{proof}

In the next result, we formally state the complexity of
checking structural controllability of $(\bA,\bB)$ using flow network $\FAB$.

\begin{lem}\label{lem:comp}
Checking structural controllability of $(\bA,\bB)$ using the maximum flow formulation has complexity $O(n^3)$.
\end{lem}
\begin{proof}
Given a structured system $(\bA, \bB)$, we know by Lemma \ref{lem:flowgraph} that complexity involved in the construction of the flow network $\FAB$ is $O(n^2)$. Finding the maximum flow in $\FAB$ has $O(|V_F||E_F|)$ computations. In the flow network corresponding to $(\bA, \bB)$, $|V_F| = O(n)$ and $|E_F| = O(n^2)$. Thus the maximum flow in $\FAB$ can be found in $O(n^3)$ computations. Thus, using Lemma \ref{lem:flowgraph} and Theorem \ref{th:maxflow}, structural controllability can be checked accurately in $O(n^3)$ computations. This completes the proof.
\end{proof}

\begin{rem}
Checking structural controllability of $(\bA, \bB)$ using the two conditions available previously in literature (i.e., SCC's and bipartite matching) has complexity $O(n^{2.5})$ and the maximum flow condition given in this paper has complexity $O(n^3)$.
\end{rem}

Using Theorem \ref{th:maxflow} and Lemma \ref{lem:comp} we infer that maximum flow problem gives another necessary and sufficient graph theoretic condition for checking structural controllability. Also, computational complexity in checking structural controllability using this condition is polynomial. The flow network constructed above is useful in two ways: (a) for checking if a given system is structurally controllable and (b) to optimize the input set for solving\remove{ minCIS and} minCCIS. The maximum flow formulation only caters (a). Our aim in this paper is to optimize the number of inputs for minCCIS. Thus we need to augment the maximum flow formulation to cater our optimization endeavour. To this end, we introduce a variant of the maximum flow problem with some additional features, called the minimum cost fixed flow problem. 
\section{Minimum Controllability Problems as Minimum Cost Fixed Flow Problem}\label{sec:mcff}
Flow networks can be used to determine structural controllability of $(\bA, \bB)$. However, the maximum flow problem may not solve minCCIS. In this section, we augment the flow network with a cost for edge usage in order to solve minCCIS. Specifically, we show that solving a minimum cost fixed flow (MCFF) problem (see for example \cite{KruNolSchWirRav:99}) in a flow network we design is equivalent to solving minCCIS. In this section, we first describe MCFF for completeness and subsequently demonstrate its utility for solving minCCIS.

Now we describe the minimum cost fixed flow problem (MCFF). Input to an MCFF problem is a directed flow network $\F(V,E)$, with vertex set $V$, edge set $E$, specified vertices $s,t$, non-negative capacities $b(e)$, non-negative costs $c(e)$ for edges $e \in E$ and flow requirement $\pl$. Then, the solution to MCFF($\pl$) is a feasible flow vector $f^\*_M$ such that $\varphi_{f^\*_M} \geqslant \pl$ and $\sum_{e \in E:f^\*_M(e) > 0} c(e) \leqslant \sum_{e \in E:f(e) > 0} c(e)$, for any feasible flow vector $f$. Thus, MCFF($\pl$) solves the following constrained optimization:

\begin{prob}\label{prob:LP}
Minimize: $\sum_{e \in E:f(e) > 0} c(e)$\\
Subject to:\\
(1)~$f$ is a feasible flow vector, and\\
(2)~$\varphi_{f} \geqslant \pl$.
\end{prob}

Any feasible solution of Problem~\ref{prob:LP} is referred to as a feasible solution to MCFF($\pl$).
Note that MCFF($\pl)$ has a feasible solution if and only if $\varphi_{f^\*} \geqslant \pl$. 
MCFF is a well studied NP-hard problem \cite{GarJoh:02}. 

To establish a relation between MCFF and minCCIS, we formulate minCCIS as an instance of MCFF such that an optimal solution $f^\*_M$ to MCFF corresponds to an optimal solution to minCCIS. 
Given a structured system $(\bA,\bB)$ and a cost vector $p_u$, such that each entry $p_u(j)$, for $j = 1,2,\ldots,m$, corresponds to the cost associated with each input, we consider flow network $\FAB$ augmented with cost vector c (referred to as $\Fc$ in the sequel) as follows:
\begin{equation}\label{eq:cost}
c(e) \leftarrow \begin{cases}
p_u(j), ~~ {\rm for}~ e = (u'_j,t), j \in \{1,2,\ldots,m\},\\
0,~~~ \mbox{otherwise.}
\end{cases} 
\end{equation}
On this flow network, we solve MCFF($q+n$). 
We have the following preliminary result.
\begin{lem}\label{lem:mcff_feas}
A structured system $(\bA,\bB)$ is structurally controllable if and only if MCFF($q+n$) has a feasible solution on $\Fc$.
\end{lem}
\begin{proof}
{\bf If part:} Here we will prove that if $(\bA,\bB)$ is structurally controllable, then $MCFF(q+n)$ has a feasible solution on $\Fc$. By Theorem \ref{th:maxflow} we know that if $(\bA,\bB)$ is structurally controllable, then the maximum flow through the network $\FAB$ is greater than or equal to $q+n$. Thus there exists a feasible flow vector $f$ of $\Fc$ such that $\varphi_f \geqslant q+n$. Thus $f$ is indeed a feasible solution to MCFF($q+n$). This completes the if part.

{\bf Only if part:} Here we will prove that if there exists a feasible solution to MCFF($q+n$) on $\Fc$, then the structured system $(\bA, \bB)$ is structurally controllable. A feasible solution to MCFF($q+n$) is a feasible flow vector $f$ such that $\varphi_f \geqslant q+n$. Since $\varphi_f \geqslant q+n$, the maximum flow vector $f^\*$ in $\Fc$ will give $\varphi_f^\* \geqslant q+n$. Thus by Theorem \ref{th:maxflow}, the structured system $(\bA, \bB)$ is structurally controllable.
\end{proof}

Henceforth, we consider a structurally controllable system $(\bA, \bB)$ with $n$ states and $q$ number of non-top linked SCC's in $\D(\bA)$. Let $f$ be
any feasible solution to MCFF($q+n$). Define,
\begin{align}
\I_f & := \{ j: f(u'_j,t) > 0 \}, \mbox{ and} \label{eq:IF}\\
c_f & := \sum_{e: f(e) > 0} c(e). 
\label{eq:CF}
\end{align}
Also define $c^\* = c_{f^\*_M}$ as the optimal cost for MCFF($q+n$) on $\Fc$.
Using \eqref{eq:IF} and~\eqref{eq:CF}, we now describe how a solution to minCCIS can be obtained from a feasible solution $f$ of MCFF($q+n$). For a given $f$, we propose to use inputs $u_j$ only if $j \in \I_f$. Following result holds.
\begin{lem}\label{lem:solution}
If $f$ is a feasible solution to MCFF($q+n$) on $\Fc$, then $(\bA,\bB_{\I_f})$ is structurally controllable and
$p(\I_f) = c_f$.
\end{lem} 
\begin{proof}
Given $f$ is a feasible solution to MCFF($q+n$) on $\Fc$. Thus $f$ is also feasible solution to MCFF($q+n$) on the flow network $\F(\bA, \bB_{\I_f})$. Now by Lemma~\ref{lem:mcff_feas} the structured system $(\bA,\bB_{\I_f})$ is structurally controllable. Now we will prove that $p(\I_f) = c_f$. This follows from Equations \eqref{eq:IF}, \eqref{eq:CF} and the cost definition given by Equation \eqref{eq:cost}.
\end{proof}

\remove{ Summarizing, a feasible solution to minCCIS is a selection of inputs $\I$ such that the system $(\bA, \bB_{\I})$ is controllable. Then the question that arises is: does a feasible solution to MCFF($q+n$) give a feasible solution to minCCIS? If so, then does optimal solution to MCFF($q+n$) when translated back give an optimal solution to minCCIS? To answer the above questions we give the following result. }

Now we prove the equivalence between minCCIS and MCFF($q+n$) through the following theorem.

\begin{theorem}\label{th:mcffformulation1}
The flow network $\Fc$ can be constructed in $O(n^2)$ computations. Also, $\I_{f^\*_M}$ is an optimal solution to minCCIS, where $f^\*_M$ is an optimal solution to MCFF($q+n$) on $\Fc$. Moreover, $p^\* = c^\*$.
\end{theorem}
\begin{proof}
From Lemma~\ref{lem:flowgraph}, we know that constructing the flow network $\FAB$ has complexity $O(n^2)$. In addition to this we define a cost vector $c$ and flow requirement $\pl$ to construct $\Fc$. Since these are of linear complexity, we conclude that complexity involved in constructing the flow network $\Fc$ is $O(n^2)$.

Let $f^\*_M$ be an optimal flow vector for MCFF($q+n$) on $\Fc$. Note that by definition $c_{f^\*_M} = c^\*$, and by Lemma~\ref{lem:solution} $c^\* = p(\I_{f^\*_M})$.
Now we show that $\I_{f^\*_M}$ is an optimal solution to minCCIS. First, we argue that $\I_{f^\*_M}$ is a feasible solution to minCCIS, i.e. $\I_{f^\*_M} \in \K$. It follows from Lemma~\ref{lem:solution} that the system $(\bA,\bB_{\I_{f^\*_M}})$ is structurally controllable. Thus, $\I_{f^\*_M} \in \K$. Suppose $\I_{f^\*_M}$ is not an optimal solution to minCCIS. Then there exists $\I\in\K$ such that $p(\I) < p(\I_{f^\*_M})$. Consider the flow network $\F(\bA,\bB_\I)$. By Theorem~\ref{th:maxflow}, there exists a feasible flow vector $f$ in  $\F(\bA,\bB_\I)$ such that $\varphi_f \geqslant q+n$. Since $\F(\bA,\bB_\I)$ is a sub-graph of $\F(\bA,\bB)$, $f$ is also a feasible flow vector in $\FAB$ with $\varphi_f \geqslant q+n$. We note that $\I_f \subseteq \I$.
Thus, from Lemma~\ref{lem:solution}, $c_f \leqslant p(\I) < p(\I_{f^\*_M}) = c^\*$. This contradicts optimality of $f^\*_M$. Thus, $\I_{f^\*_M}$ is an optimal solution to minCCIS. Finally, $p^\* = c^\*$ follows from Lemma~\ref{lem:solution} and optimality of $\I_{f^\*_M}$ for minCCIS.
\end{proof} 

Thus given an instance of minCCIS we construct $\Fc$ and reduce it to an MCFF as discussed. After solving MCFF($q+n$), we get an optimal flow $f^\*_M$. From $f^\*_M$, we get back the corresponding optimal solution to minCCIS, $\I^\* =  \I_{f^\*_M}=\{j: f^\*_M(u'_j,t) > 0 \}$, i.e., the minimum cost incurring set of inputs selected under $f^\*_M$. 
Unfortunately, MCFF is also a known NP-hard problem. Thus, optimal solution $f^\*_M$ may not be obtained in polynomial time complexity unless $P=NP$. However, it is a well studied problem as it relates to many fields including job-shop scheduling, transportation network and computer networks \cite{MagWon:84}. For MCFF, approximation algorithm in addition to many good heuristics exist \cite{KhaFuj:91}. The commonly used approaches for approximating MCFF include local search of adjacent extreme flows \cite{GalSod:79}, \cite{Yag:71}, dynamic programming \cite{Zan:68}, \cite{EriMonVei:87} and branch and bound technique \cite{FloRob:71}, \cite{Sol:74}. We can potentially use these existing algorithms to obtain approximate solution to minCCIS. However to do this, we need to show that an approximate solution to MCFF yields an approximate solution to minCCIS. We establish this in the following result.
\begin{theorem}\label{th:epsilon1}
Let $f$ be a feasible solution to MCFF($q+n$) on $\Fc$. Then for any $\epsilon \geqslant 1$, $c_f \leqslant  \epsilon\,c^\* $ implies that $p(\I_f) \leqslant \epsilon\,p^\* $.
\end{theorem}
\begin{proof}
The result immediately follows from Lemma~\ref{lem:solution} and Theorem~\ref{th:mcffformulation1}.
\end{proof}
Note that a feasible solution $f$ that satisfies the condition in Theorem~\ref{th:epsilon1} is called as an $\epsilon$-optimal solution.
In the next section, we obtain an approximation algorithm for MCFF($q+n$) on $\Fc$.
\section{Approximation Algorithm for MCFF on $\Fc$}\label{sec:approx}
MCFF over general graphs are shown to be hard to approximate \cite{KruNolSchWirRav:99}, \cite{EveKorSla:05}. Specifically, the following result is known.

\begin{prop} [\cite{KruNolSchWirRav:99}, Theorem 17]\label{prop:inapprox}
MCFF is strongly NP-hard even on bipartite graphs. Unless $NP \subseteq DTIME(N^{O({\rm log~}{\rm log}N)})$, for any $\epsilon > 0$ there is no approximation
algorithm for MCFF with a performance of $(1-\epsilon){~\rm ln}\pl$, where $\pl$ is the given flow value to be achieved.
\end{prop}
 However, a $\pl$-approximate solution to MCFF($\pl$) is given in \cite{KruNolSchWirRav:99}, \cite{AssEmaFarYazZar:14}. The algorithm given in \cite{KruNolSchWirRav:99} uses a min-cost max-flow algorithm. The approximation in \cite{AssEmaFarYazZar:14} uses a primal-dual formulation and  has complexity $O(|V||E|^2{\rm log}(|V|^2/|E|))$. We give a polynomial complexity $\Delta$-approximate solution to MCFF($q+n$) on $\Fc$, where $\Delta$ is the maximum in-degree of nodes $u'_j$'s in $\Fc$. Note that $1 \leqslant \Delta \leqslant q+1$. Next we elaborate our approach. 
 
Consider a flow network $\Fc$ and define the following Linear Program (LP):  
\begin{prob}\label{prob:opt}

\noindent
Minimize: $\sum_{e \in E_F} c(e)f(e)$

\noindent
Subject to:

(1)~$f$ is a feasible flow vector, and

(2)~$\varphi_{f} \geqslant q+n$.
\end{prob}

Problem~\ref{prob:opt} is a well studied flow problem in literature, known as the {\it minimum cost flow problem} (MCFP) \cite{AhuMagOrl:93}. 
Note that the key difference between MCFF and MCFP is that in the former cost incurred does not depend on the flow through an edge, rather it depends only on whether the edge is used; however in the latter the cost increases linearly with the flow through the edge.  
MCFP can be solved in polynomial time with complexity $O(\ell^4\,{\rm log}\,\ell)$, where $\ell$ denotes the number of vertices in the flow network \cite{Orl:93}. Let the value of the objective function in Problem~\ref{prob:opt} for a feasible flow vector be $C_f$. Also, let $\c$ denotes the minimum value of the objective function of Problem~\ref{prob:opt}. Also, let $\flp$ be the corresponding optimal flow vector, i.e. $\c \leqslant C_f$ for any feasible solution $f$ of the LP. Following preliminary result holds as a direct consequence of [\cite{AhuMagOrl:93}, Theorem~9.8, pp.~318].

\begin{lem}\label{lem:lpFlowInt}
For every $e\in E_F$, $\flp(e) \in \mathbb{Z}$.
\end{lem}

In the following result, we obtain a relation between the optimal value $\c$ of the LP and the optimal cost $c^\*$ of MCFF($q+n$).

\begin{lem}\label{lem:cAndc}
Following holds: $\c \leqslant \Delta \, c^\*$, where $\Delta$ is the maximum in-degree for nodes in $\{u'_j\}_{j=1}^m$.
\end{lem}
\begin{proof}
Note that $b(e) = 1$ for every $e \in E_F \setminus \{(u'_j,t)\}_{j=1}^m$. Thus, the total flow carried by any $e \in E_F$ is at most $\Delta$ under any feasible flow vector $f$. Hence, we have the following:
\begin{align}
\c & = \sum_e \flp(e) c(e), \nonumber \\
& \leqslant  \sum_e f^\*_M(e) c(e), \label{eq:cc1} \\
& = \sum_{e : f^\*_M(e) > 0} f^\*_M(e) c(e), \nonumber \\
& \leqslant \Delta \, \sum_{e : f^\*_M(e) > 0} c(e), \label{eq:cc2} \\
& = \Delta \, c^\*. \nonumber
\end{align}
Equation \eqref{eq:cc1} follows as $f^\*_M$ is a feasible solution to the LP. Equation \eqref{eq:cc2} follows as $f(e) \leqslant \Delta$ for every $e$ and feasible flow vector $f$. 
\end{proof}
In the following result, we obtain a relation between the optimal value $\c$ of the LP and the cost of the inputs selected under $\flp$ denoted as $c_{\flp}$.
\begin{lem}\label{lem:coptAndc}
Following holds: $c_{\flp} \leqslant \c$.
\end{lem}
\begin{proof}
The cost of the inputs selected under $\flp$ 
\begin{align}
c_{\flp} & = \sum_{e: \flp(e) > 0} c(e), \nonumber \\
& \leqslant  \sum_{e: \flp(e) > 0} \flp(e) c(e), \label{eq:cc3} \\
& = \c, \nonumber 
\end{align}
Equation \eqref{eq:cc3} follows as if $\flp(e) > 0$, then $\flp(e) \geqslant 1$ by Lemma~\ref{lem:lpFlowInt}.
\end{proof}

Following key result is an immediate consequence of Lemmas~\ref{lem:cAndc} and~\ref{lem:coptAndc}.

\begin{theorem}\label{th:approximation}
The flow vector $\flp$ is a $\Delta$-approximate solution of MCFF($q+n$), i.e. $c_\flp \leqslant \Delta \, c^\*$.
\end{theorem}
\begin{proof}
Note from Lemmas~\ref{lem:cAndc} and~\ref{lem:coptAndc} that $c_\flp \leqslant \c \leqslant \Delta \, c^\*$. This proves the required.
\end{proof}

\begin{rem}\label{rem:scc}
Note that the number of non-top linked SCC's is at most $n$. Thus, in the worst case $\Delta = O(n)$. This corresponds to states being decoupled. However, in practical systems the states interact and as a result the number of non-top linked SCC's may be much smaller than $n$. In such cases, the above algorithm may give a tighter approximation.  
\end{rem}
 
There exist various polynomial algorithms for solving Problem~\ref{prob:opt}. The known algorithms include capacity scaling algorithm, cost scaling algorithm, double scaling algorithm, minimum mean cycle-cancelling algorithm, repeated capacity scaling algorithm and enhanced capacity scaling algorithm (see \cite{AhuMagOrl:93} and references therein for more details). The best strongly polynomial algorithm runs in  $O(\ell^4\,{\rm log\,}\ell)$ in a generic flow network with $\ell$ nodes \cite{Orl:93}. However, because of the special structure of the flow network $\Fc$, Problem \ref{prob:opt} can be solved using a simpler algorithm that incorporates a minimum weight perfect matching and a greedy scheme. We describe the pseudo code of this two stage procedure in Algorithm~\ref{alg:twostage}.
\begin{algorithm}[t]
  \caption{Pseudo code for solving Problem \ref{prob:opt}
  \label{alg:twostage}}
  \begin{algorithmic}
\State \textit {\bf Input:} Structured system $(\bA, \bB)$, input cost vector $p_u$ and flow network $\Fc$ 
\end{algorithmic}
  \begin{algorithmic}[1]

  \State  Construct $\B(\bA, \bB)$ \label{step:bip}
  \State  For each edge $e$ define weight
\State $w(e) \leftarrow \begin{cases}
0 , ~~ {\rm for}~ e = (x_r,x'_k),\\
p_u(j),~~ {\rm for}~ e = (u_j,x'_k). 
\end{cases}$ \label{step:weight}
\State Find minimum weight perfect matching of $\B(\bA, \bB)$ under weight function $w$, say $M_A$\label{step:match}
\For {$i = 1$ to $q$}
\State $u(\cN_i) \in \arg\min_{u_j: (\cN_i, u'_j) \in E_F}p_u(j)$\label{step:greedy}
\State $\S_A \leftarrow \{(\cN_i, u'(\cN_i)) \}_{i=1}^q$\label{step:edgeset}
\EndFor
\end{algorithmic}
\begin{algorithmic}
\State \textit{\bf Output:} Flow vector $f_A$ constructed using Algorithm~\ref{alg:flow_construct} with inputs $M_A$ and $\S_A$.
  \end{algorithmic}
\end{algorithm} 

In the first stage of Algorithm~\ref{alg:twostage}, we run a minimum weight perfect matching algorithm on the system bipartite graph $\B(\bA, \bB)$ with weights defined as shown in Step~\ref{step:weight}. Let $M_A$ be a matching obtained as solution of this stage (see Step~\ref{step:match}). In stage two, we perform a greedy selection to connect all the non-top linked SCC's to some input. To achieve this for all $\cN_i$'s, $i \in \{1,\ldots,q\}$, we greedily assign the least cost input which has an edge to some state in $\cN_i$ (see Step~\ref{step:greedy}). Let the least cost input corresponding to $\cN_i$ be $u(\cN_i)$. We define $\S_A = \{(\cN_i,u'(\cN_i))\}_{i=1}^q$ (see Step~\ref{step:edgeset} ). 
Finally, we use Algorithm~\ref{alg:flow_construct} to construct a flow vector $f_A$ based on $M_A$ and $\S_A$.

\begin{algorithm}
  \caption{Pseudo code for constructing flow vector $f$ from perfect matching $M$ and set of edges $\S=\{(\cN_i, u'(\cN_i)) \}_{i=1}^q$
  \label{alg:flow_construct}}
  \begin{algorithmic}
\State \textit {\bf Input:} Perfect matching $M$ and edge set $\S$
\end{algorithmic}
  \begin{algorithmic}[1]
\State $\X_M \leftarrow \{x_r: (x_r, x'_k) \in M\} $
\State $\U_M \leftarrow \{u_j: (u_j, x'_k) \in M\} $ 
\State $\U_S \leftarrow \{u_j: (\cN_i, u'_j) \in \S\} $ 
\State  Define the flow vector $f$ as
\State $f(e) \leftarrow \begin{cases}

1,~~{\rm for}~e \in \{(s,x'_k),(x'_k, x_r), (x_r, t)\},\\
~{\rm for~} k=1,\ldots,n~{\rm and~} x_r \in \X_M,\\

1,~~{\rm for}~e \in \{(s,x'_k), (x'_k, u_j), (u_j, u'_j)\},\\
~{\rm for~} k=1,\ldots,n~{\rm and~} u_j \in \U_M,\\

1,~~{\rm for}~e \in \{(s,\cN_i,), (\cN_i, u'_j)\},\\
~{\rm for~} i=1,\ldots,n~{\rm and~} u_j \in \U_S, \\
\sum_{k=1}^n \sum_{j=1}^m \mathbb{I}_{\{(x'_k,u_j)\in M\}} + \sum_{i=1}^q\sum_{j=1}^m \mathbb{I}_{\{(\cN_i,u'_j)\in \S\}},\\
 {\rm for}~e=(u'_j,t), ~{\rm for~} j=1,\ldots,m.

\end{cases}$ \label{step:flow}
\end{algorithmic}
\begin{algorithmic}
\State \textit{\bf Output:} Flow vector $f$
\end{algorithmic}
\end{algorithm}

We prove the optimality of the constructed flow vector $f_A$ after stating the following supporting lemmas.

\begin{lem}\label{lem:LP_feasible_flow}
Given any valid inputs $M$ and $\S$ of Algorithm~\ref{alg:flow_construct}, let $f$ denote the output flow vector. Then $f$ is a feasible solution to the LP given in Problem~\ref{prob:opt}. Moreover, the value
\begin{align}\label{eq:Cf}
C_f = \sum_{k=1}^n \sum_{j=1}^m p_u(j)\mathbb{I}_{\{(x'_k,u_j)\in M\}} + \sum_{i=1}^q\sum_{j=1}^m p_u(j)\mathbb{I}_{\{(\cN_i,u'_j)\in \S\}}, \end{align}
where $\mathbb{I}_{\cal A}$ is the indicator function of ${\cal A}$.
\end{lem}
\begin{proof}
From the construction of $f$ as per Algorithm~\ref{alg:flow_construct}, it follows that $f$ satisfies both flow conservation and capacity constraints in $\Fc$. Thus, the flow vector $f$ is feasible in $\Fc$.  Moreover, note that $f(e) = 1$ for every $e$ that emanates from the source vertex $s$. Hence $\varphi_f = q+n$. This shows that $f$ is a feasible solution to the LP given in Problem~\ref{prob:opt}.

Now, note that since costs are non-zero only for the edges between $u'_j$ and $t$. Also, $c((u'_j,t))=p_u(j)$. Thus,
\[C_f = \sum_{j=1}^m p_u(j) f((u'_j,t)).\] The flow value $f((u'_j,t))$ equals the sum of the flows coming from edges $(u_j,u'_j)$ and $(\cN_i,u'_j)$. Note that the second term in \eqref{eq:Cf} corresponds to the total cost contributed by the flow from edges $(\cN_i,u'_j)$. Now,
the flow on $(u_j,u'_j)$ is greater than zero then it has to come from some edge $(x'_k,u_j)\in M$. Thus, the first term in \eqref{eq:Cf} corresponds to the total cost contributed by the flow from edges $(u_j,u'_j)$. This proves the required.
\end{proof}

\begin{lem}\label{lem:algoflow}
The sets $M_A$ and $\S_A$ given by the Algorithm~\ref{alg:twostage} satisfies the following:

\noindent
(1)~For any perfect matching $M$ of $\B(\bA,\bB)$,
\begin{align*}
\sum_{k=1}^n \sum_{j=1}^m p_u(j)\mathbb{I}_{\{(x'_k,u_j)\in M_A\}} \leqslant \sum_{k=1}^n \sum_{j=1}^m p_u(j)\mathbb{I}_{\{(x'_k,u_j)\in M\}}, \mbox{ and}
\end{align*}

\noindent
(2)~For any set $\S=\{(\cN_i,y'_i):y_i\in V_U, (\cN_i,y'_i) \in E_F\}_{i=1}^q$, 
\begin{align*}
\sum_{i=1}^q\sum_{j=1}^m p_u(j)\mathbb{I}_{\{(\cN_i,u'_j)\in \S_A\}} \leqslant \sum_{i=1}^q\sum_{j=1}^m p_u(j)\mathbb{I}_{\{(\cN_i,u'_j)\in \S\}}.
\end{align*}
\end{lem}
\begin{proof}
The result is an immediate consequence of the way in which the Algorithm~\ref{alg:twostage} constructs $M_A$ and $\S_A$.
\end{proof}

Using the above algorithm we prove the following result.

\begin{theorem}\label{th:two_stage}
The following holds: $C_{f_A} = \c$.
\end{theorem}

\begin{proof}
First, observe from Lemma~\ref{lem:LP_feasible_flow} that  ${f_A}$ is a feasible solution to the LP described in Problem~\ref{prob:opt}. Thus, $C_{f_A} \geqslant \c$. Now, we get the result if we can show $\c \geqslant C_{f_A}$.  Let $\flp$ be the optimal flow vector. Define, the following sets:
\begin{align*}
M^\* &= \{(x'_j,y_j) : \flp((x'_j,y_j)) > 0, y_j\in V_X\cup V_U\}, \\
\S^\* &= \{(\cN_i,u'_j) : \flp(\cN_i,u'_j) > 0 \mbox{ for some } i\in\{1,\ldots,q\}\}.
\end{align*}
Note that $M^\*$ is a perfect matching in $\B(\bA,\bB)$. Also, $\S^\*$ has an outgoing edge from every non-top linked SCC to some input $u'_j$. Note that $\flp$ can be thought as a flow vector constructed from $(M^\*, \S^\*)$
using Algorithm~\ref{alg:flow_construct}. Now the result follows from \eqref{eq:Cf} and Lemma~\ref{lem:algoflow}.
\end{proof}

The following result quantifies computational complexity of Algorithm~\ref{alg:twostage}.
\begin{lem}\label{lem:two_stage_complexity}
Algorithm~\ref{alg:twostage} has complexity $O(n^{3})$.
\end{lem}

\begin{proof}
Stage one of Algorithm~\ref{alg:twostage} where we solve a minimum weight perfect matching has complexity $O(n^{3})$. Stage two of Algorithm~\ref{alg:twostage} where a greedy scheme is employed to connect all non-top linked SCC's has $O(n^2)$ complexity, since $q = O(n)$ and $m=O(n)$. From the two stages we get $M_A$ and $\S_A$. Construction of flow vector $f_A$ using $M_A$ and $\S_A$ given in Algorithm~\ref{alg:flow_construct} is of linear complexity. Thus, Algorithm~\ref{alg:twostage} has complexity $O(n^{3})$.
\end{proof}
\section{Special Cases} \label{sec:cases}
In this section, we discuss two special cases of Problem~\ref{prob:two}. Firstly, consider the case where there is a perfect matching in the state bipartite graph $\B(\bA)$. Note that Problem~\ref{prob:two} is NP-hard over this special class of structured systems also \cite{PeqSouPed:15}. We give the following approximation result for this class of systems.

\begin{cor}\label{cor:class_matching}
Let there exists a perfect matching in the bipartite graph $\B(\bA)$. Then, the solution of Problem~\ref{prob:opt}, $\flp$, is a \mbox{$(\Delta-1)$-approximate} solution to MCFF($q+n$), i.e. $c_\flp \leqslant (\Delta-1) \, c^\*$.
\end{cor}

\begin{proof}
If there exists a perfect matching in the bipartite graph $\B(\bA)$ the edges $(u_j, u'_j)$ does not carry any flow and thus these edges can be removed from $\Fc$. Thus the effective in-degree of the nodes $\{u'_1,\ldots,u'_m\}$ is atmost $\Delta - 1$. Thus using Theorem~\ref{th:approximation} we get the required bound.
\end{proof}
The second class of structured systems considered are the ones where the digraph $\D(\bA)$ is irreducible. In such a case Problem~\ref{prob:two} is no longer NP-hard. Here, we consider two cases: (a)~$\B(\bA)$ has a perfect matching, and (b)~$\B(\bA)$ does not have a perfect matching.
\begin{cor}\label{cor:class_scc}
If $\D(\bA)$ is irreducible, then an optimal solution to Problem~\ref{prob:opt}, $\flp$, gives an optimal solution to MCFF($q+n$). Also, $c_\flp = c^\*$.
\end{cor}
\begin{proof}
In case~(a) the structured system requires just one input to make it structurally controllable. Thus, the least cost input is the optimal solution. In case~(b), since the digraph $\D(\bA)$ is irreducible, any input will make all the states in the system accessible. Thus, we can simply remove the nodes corresponding to the non-top linked SCC's $\{\cN_1, \ldots,\cN_q\}$ and the edges associated with it from $\Fc$. Thus the effective in-degree of the nodes $\{u'_1,\ldots,u'_m\}$ is $1$. Thus using Theorem~\ref{th:approximation} we get the required. 
\end{proof}

\begin{rem}\label{rem:obs1}
By duality between controllability and observability in linear time invariant systems structural observability of $(\bA, \bC)$ is equivalent to structural controllability of $(\bA^T, \bC^T)$. Thus all results discussed in this paper is applicable to minimum cost constrained output selection, where given $\bA \in \{0, \*\}^{n \times n}$, $\bC \in \{0, \*\}^{p \times n }$, such that $(\bA, \bC)$ is structurally observable and a cost vector $p_y(j)$, $j =1,\ldots, p$, where $p_y(j)$ denote the cost of sensing $j^{\rm th}$ output, our aim is to find the minimum cost incurring output set $\widetilde{\I}$ such that the system $(\bA, \bC_{\widetilde{\I}})$ is observable, where $\bC_{\widetilde{\I}}$ consists of only those rows of $\bC$ whose indices are present in the set $\widetilde{\I}$.
\end{rem}

\begin{rem}\label{rem:discrete}
From linear systems theory, we know that the controllability criterion for continuous and discrete linear time invariant systems is same. That is, rank$~[B~AB~ \cdots A^{n-1}B] = n$, where $A,B$ are system matrices and $n$ is the system dimension. Thus all the analysis and results obtained for minCIS and minCCIS for continuous systems are applicable to minCIS and minCCIS problems of discrete systems.
\end{rem}
\section{Conclusion}\label{sec:conclu}
This paper addresses two structural controllability problems, namely minimum constrained input selection (minCIS) and minimum cost constrained input selection (minCCIS). Both of these problems are known to be to be NP-hard (see \cite{PeqSouPed:15}). Thus there do not exist polynomial algorithms for solving these unless P=NP. In this paper, we first provide a new graph theoretic necessary and sufficient condition based on flow networks for checking structural controllability (see Theorem \ref{th:maxflow}). Then we give a polynomial reduction of minCCIS to a NP-hard variant of the maximum flow problem, the minimum cost fixed flow problem (MCFF). We showed that an optimal solution to MCFF problem corresponds to an optimal solution to the minCCIS problem (see Theorem \ref{th:mcffformulation1}). We also showed that approximation schemes available for solving MCFF can be used to solve minCCIS (Theorem~\ref{th:epsilon1}). Using the special structure of the flow network constructed from the structured system $(\bA, \bB)$, we propose an approximation algorithm to solve minCCIS. In our main result we give a polynomial algorithm that obtains a $\Delta$-approximate solution to minCCIS (see Theorem \ref{th:approximation}). All the results for minCCIS directly applies to minCIS, since it is a special case of minCCIS where the costs associated with inputs are non-zero and uniform. In this work, both the problems are considered in their full generality without any assumptions on the system.
\bibliographystyle{myIEEEtran}  
\bibliography{minCCIS_Approx} 
\end{document}